\documentclass[a4paper]{article}

\usepackage{amsmath}
\usepackage{amssymb}
\usepackage{amsfonts}
\usepackage{enumerate}
\usepackage{vmargin}

\newcommand{\tens}{\otimes}

\newcommand{\qed}{{\vrule height5pt width5pt depth1pt}}

\newcommand{\bB}{{\mathbb{B}}}

\newcommand{\bM}{{\mathbb{M}}}

  \newcommand{\A}{{\mathcal{A}}}
  \newcommand{\B}{{\mathcal{B}}}
  \newcommand{\C}{{\mathcal{C}}}
  \newcommand{\D}{{\mathcal{D}}}

\renewcommand{\L}{{\mathcal{L}}}
  \newcommand{\M}{{\mathcal{M}}}
  \newcommand{\N}{{\mathcal{N}}}

\renewcommand{\S}{{\mathcal{S}}}
  \newcommand{\T}{{\mathcal{T}}}
  \newcommand{\U}{{\mathcal{U}}}

  \newcommand{\X}{{\mathcal{X}}}
  \newcommand{\Y}{{\mathcal{Y}}}

\renewcommand{\phi}{\varphi}
\newcommand{\upchi}{{\raise.35ex\hbox{\ensuremath{\chi}}}}

\newcommand{\id}{{\operatorname{id}}}

\newcommand{\m}{\operatorname{m}}

\newtheorem{thm}{Theorem}[section]
\newtheorem{defi}[thm]{Definition}
\newtheorem{prop}[thm]{Proposition}
\newtheorem{cor}[thm]{Corollary}
\newtheorem{lemma}[thm]{Lemma}

\newcounter{remark}
\newenvironment{rk}{{\noindent
\bf Remark \addtocounter{thm}{1}\arabic{section}.\arabic{thm} \ }}
{
\smallskip
}

\newenvironment{proof}[1][]{\noindent {\it Proof #1}: }{\hbox{~} \hfill\qed
\smallskip
}

\title{A noncommutative Amir-Cambern theorem for von Neumann algebras and nuclear $C^*$-algebras}
\date{}
\author{\'Eric Ricard \and Jean Roydor}

\begin{document}

\maketitle

\begin{abstract} We prove that von Neumann algebras and separable nuclear $C^*$-algebras are stable for the Banach-Mazur cb-distance. A technical step is to show that unital almost completely isometric maps between $C^*$-algebras are almost multiplicative and almost selfadjoint. Also as an intermediate result, we compare the Banach-Mazur cb-distance and the Kadison-Kastler distance. Finally, we show that if two $C^*$-algebras are close enough for the cb-distance, then they have at most the same length.

\end{abstract}

\section{Introduction}

This note concerns perturbations of operator algebras as operator
spaces, more precisely perturbations relative to the Banach-Mazur
cb-distance. In \cite{Pi0}, G. Pisier introduced the Banach-Mazur
cb-distance (or cb-distance in short) between two operator spaces
$\X,\Y$:

$$d_{cb}(\X,\Y)=\inf \big\{ \| T \|_{cb} \| T^{-1} \|_{cb}
\big\}, $$ where the infimum runs over all possible linear
completely bounded isomorphisms $T:\X \to \Y$. This extends naturally the
classical Banach-Mazur distance for Banach spaces when these are
endowed with their minimal operator space structure (in particular, the Banach-Mazur distance and the cb-distance between two $C(K)$-spaces coincide). For background on
completely bounded maps and operator space theory the reader is
referred to \cite{BLM}, \cite{ER1}, \cite{Pa} and \cite{P}.\\
 Let us
recall the generalization of Banach-Stone theorem obtained independently by 
 D. Amir and M. Cambern (see \cite{A}, \cite{Ca}): if the Banach-Mazur
distance between two $C(K)$-spaces is strictly smaller than 2, then
they are $*$-isomorphic (as $C^*$-algebras). Actually, this is also true for spaces of continuous functions vanishing at infinity on locally compact Hausdorff spaces. One is tempted to extend
the Amir-Cambern Theorem to noncommutative $C^*$-algebras. In \cite{K}, R. Kadison described isometries between $C^*$-algebras, in particular the isometric structure of a $C^*$-algebra only determines its Jordan structure, hence to recover the $C^*$-structure we need a priori assumption on the cb-distance (not only on the classical Banach-Mazur distance). Here, we prove:

\smallskip

\noindent \textbf{Theorem A.}   \textit{Let $\A$ be a separable nuclear $C^*$-algebra or a von Neumann algebra, then there exists an $\varepsilon_0 >0$ such that for any $C^*$-algebra $\B$, the inequality 
$d_{cb}(\A,\B)<1+\varepsilon_0$ implies that $\A$ and $\B$ are $*$-isomorphic.\\
When $\A$ is a separable nuclear $C^*$-algebra, one can take $\varepsilon_0=3.10^{-19}$. When $\A$ is a von Neumann algebra, $\varepsilon_0=4.10^{-6}$ is sufficient.}

\smallskip

Such a result can not be extended to all unital
$C^*$-algebras, see Corollary \ref{evian} below for a counter-example (derived from \cite{CC}) involving nonseparable $C^*$-algebras.\\
The proof of Theorem A is totally different
from the commutative case, the cb-distance concerns only the operator
space structure, hence the basic idea is to gain the algebraic
structure. It is known that unital completely isometric linear isomorphisms
between operator algebras are necessarily multiplicative (see Theorem 4.5.13 \cite{BLM}). Therefore, one wants to
prove that almost completely isometric maps are almost multiplicative; in the sense that the defect of multiplicativity has small cb-norm as a bilinear map. We manage to check this by an ultraproduct
argument (see Proposition \ref{gene}) but without any explicit control
on the defect of multiplicativity. When maps are between $C^*$-algebras, we can drop the
`unital' hypothesis and show that the unitization of an almost
completely isometric map between $C^*$-algebras is almost
multiplicative with explicit bounds.
Consequently, starting from a linear cb-isomorphism with small bound
between $C^*$-algebras, one can define a new multiplication on each
of them close to the original ones. Then, in the spirit of \cite{J1} or \cite{RT},
we use the vanishing of the second and third completely bounded
Hochschild cohomology groups of an operator algebra over itself to
establish a strong stability property under perturbation by close
multiplications (see Proposition \ref{P:3}). It is crucial to work
with the completely bounded cohomology here, because we can exploit the
deep result that every completely bounded cohomology group of a
von Neumann algebra over itself vanishes (see \cite{SS}), this is
unknown for the bounded cohomology.  This allows us to conclude for
von Neumann algebras.\\
 For separable nuclear $C^*$-algebras, the
strategy is different, because vanishing of completely bounded
cohomology groups is not available. First, we
compare the cb-distance $d_{cb}$ and the completely bounded Kadison-Kastler distance
$d_{KK,cb}$ (see the definition below):

\smallskip

 \noindent \textbf{Theorem B.}  \textit{There exists a constant $K>0$, such that for any $C^*$-algebras $\A$ and $\B$:
$$d_{KK,cb}(\A,\B)\leq K \sqrt{\ln d_{cb}(\A,\B)}.$$ 
One can choose $K=3620$, when $d_{cb}(\A,\B)< 1 + 10^{-7}.$}

\smallskip
In order to prove this theorem, we need to control explicitly the defect of selfadjointness of a unital almost completely isometric map.
Then we will use stability of 
separable nuclear $C^*$-algebras for the Kadison-Kastler distance, this is a major result in perturbation theory. Let us recall from \cite{C5} this result more precisely. As usual $H$ denotes a Hilbert space and $\bB(H)$ its bounded linear endomorphisms.
 Let $\A,\, \B$ be subalgebras of $\bB(H)$, the Kadison-Kastler distance
between $\A$ and $\B$ inside $\bB(H)$ is
$$d_{KK, H}(\A,\B)=d_{\bB(H)} \big (Ball(\A),Ball(\B) \big ),$$  where $d_{\bB(H)}$ denotes the Hausdorff distance and $Ball(\A)$ (respectively $Ball(\B)$) denotes the unit ball of $\A$ ($\B$ respectively).
More generally, for two $C^*$-algebras $\A$ and $\B$,  the Kadison-Kastler distance and its completely bounded version are defined as:   
$$ d_{KK}(\A,\,\B)=\inf_{\pi,\rho,H}  \big\{d_{KK, H} \big ( \pi(\A),\rho(\B) \big ) \big \},$$
$$ d_{KK,cb}(\A,\,\B)=\inf_{\pi,\rho,H} \big\{ \sup_n \big\{ d_{KK, \ell^2_n \otimes H} \big ( (\id_{\mathbb{M}_n} \otimes \pi)(\A), (\id_{\mathbb{M}_n} \otimes \rho)(\B) \big ) \big \} \big \},$$
where the infimum runs over all faithful unital $*$-representations
$\pi:\A\to \bB(H),\,\rho:\B\to \bB(H)$ on the same Hilbert space.
The main result of \cite{C5} is: for any $\gamma<42^{-1}.10^{-4}$, if
$\A$ is a separable nuclear $C^*$-algebra and $\B$ is another
$C^*$-algebra, if $d_{KK}(\A,\B)\leq \gamma$ then $\A$ and $\B$ are
$*$-isomorphic. Therefore, it is clear that the $C^*$-case of Theorem
A is a corollary of this last result and our Theorem B.

 We already mentioned that an Amir-Cambern type theorem is false for any
$C^*$-algebras, however we can try to prove that some $C^*$-algebraic
invariants are preserved under perturbation relative to the
cb-distance.  The notion of length of an operator algebra has been
defined by G. Pisier in \cite{Pi1} in order to attack the Kadison
similarity problem (he proved that a $C^*$-algebra has finite length
if and only if it has the Kadison similarity property). In \cite {C6} Theorem 4.4,
the authors proved that having finite length is a property which is stable under perturbation for
the Kadison-Kastler distance. Here, we prove that if two $C^*$-algebras are close enough for the cb-distance, then they have at most the same length.

\smallskip

 \noindent \textbf{Theorem C.}  \textit{ Let $K\geq 1$ and 
$\ell \in \mathbb{N}\backslash \{0\}$ fixed but arbitrary constants. 
If
$\A$ and $\B$ are unital $C^*$-algebras with
$d_{cb}(\A,\B)<1+10^{-4}4^{-\ell}K^{-2}$ and $\A$ has length at most $\ell$
and length constant at most $K$, then $\B$ has length at most $\ell$.
}

\section{Almost completely isometric maps}

 This section starts with few technical lemmas relating algebraic
 properties to norm estimates in operator algebras. We next use them
 to study almost completely isometric isomorphisms. We implicitly refer to
 \cite{ER1}, \cite{Pa} and  \cite{P} for basic notions around operator spaces.

It is well known that unital
completely isometric isomorphisms between operator algebras are
necessarily multiplicative (see Theorem 4.5.13 \cite{BLM}). Hence one can hope that unital
almost completely isometric bijections are almost multiplicative and almost selfadjoint. This can be checked easily by an ultraproduct argument, but the important point is to control explicitly the defect of multiplicativity and the defect of selfadjointness.

When $T:\A\to \B$ is a map between two operator algebras, the defect of multiplicativity of $T$ is denoted by
$T^\vee$. It consists of the bilinear map $T^\vee:\A^2\to \B$ given by
$T^{\vee}(a,b)=T(ab)-T(a)T(b)$. As usual when dealing with bilinear maps (see
section 1.4 in \cite{SS}), the completely bounded norm of $T^\vee$ is
the cb-norm of the induced linear map $T^\vee : \A\tens_h \A\to
\B$ on the Haagerup tensor product.

Given a cb map $T:\S\to \T$ between two operator systems, we use the notation 
$T^\star$ for the map defined on $\S$ by $T^\star(x)=T(x^*)^*$. 
We call defect of selfadjointness the linear map $T-T^\star$. 

\begin{prop}\label{gene} For any $\eta>0$, there exists $\rho \in ]0,1[$ such that for any unital operator algebras $\A$, $\B$, for any unital cb-isomorphism $T:\A \to \B$,
$\| T \|_{cb}\leq 1+\rho$ and $\| T^{-1} \|_{cb}\leq 1+\rho$ imply $\| T^{\vee} \|_{cb} < \eta$.
\end{prop}
\begin{proof}
Suppose the assertion is false. Then there exists $\eta_0>0$ such that
for every positive integer $n \in \mathbb{N}\backslash \{0\}$, there
is a unital cb-isomorphism $T_n:\A_n \to \B_n$ between some unital
operator algebras satisfying $$\| T_n \|_{cb}\leq 1+\frac{1}{n}, ~~\|
T_n^{-1} \|_{cb}\leq 1+\frac{1}{n} ~~\hbox{and}~~ \| T_n^{\vee }
\|_{cb} \geq \eta_0.$$ Let $\U$ be a nontrivial ultrafilter on
$\mathbb{N}$, let us denote $\A_\U$ (resp. $\B_\U$) the ultraproduct
$\Pi_n \mathbb{K}^1 \otimes_{\min} \A_n /\U$ (resp. $\Pi_n
\mathbb{K}^1 \otimes_{\min} \B_n /\U$), here $\mathbb{K}^1$ denotes
the unitization of the $C^*$-algebra of all compact operators on
$\ell_2$. Then $\A_\U$ (resp. $\B_\U$) is a unital operator algebra
(see \cite{BLM}). Now consider $T_\U: \A_\U \to \B_\U$ the
ultraproduct map obtained from the $\id_{\mathbb{K}^1} \otimes
T_n$'s. Hence $T_\U$ is a unital surjective linear complete isometry between
operator algebras, so $T_\U$ is multiplicative (see Theorem 4.5.13
\cite{BLM}) hence $T_\U^\vee=0$. 

This contradicts the hypothesis for all $n$, $\|
T_n^{\vee } \|_{cb} \geq \eta_0$. Indeed $\| T_n^{\vee} \|_{cb}=\|
(\id_{\mathbb{K}^1} \otimes T_n)^{\vee } \|$, so there are $u_n,\,v_n$
in the closed unit ball of $\mathbb{K}^1 \otimes_{\min} \A_n$ such
that
$$\big\| (\id_{\mathbb{K}^1} \otimes T_n) (u_nv_n) -(\id_{\mathbb{K}^1} \otimes T_n) (u_n)(\id_{\mathbb{K}^1} \otimes T_n) (v_n) \big\| \geq \eta_0,$$
which implies that $$\big\| T_\U (\dot{u} \dot{v})-T_\U (\dot{u})T_\U (\dot{v}) \big\| \geq \eta_0$$
(where $\dot{x}$ denotes the equivalence class of $(x_{n})_n$ in $\A_\U$). 
\end{proof}

A similar proof gives

\begin{prop}\label{gene2} For any $\eta>0$, there exists $\rho \in ]0,1[$ such that for any operator systems $\S$, $\T$, for any unital cb-isomorphism $T:\S \to \T$,
$\| T \|_{cb}\leq 1+\rho$ and $\| T^{-1} \|_{cb}\leq 1+\rho$ imply $\| T -T^\star\|_{cb} < \eta$.
\end{prop}

We turn to quantitative versions of the previous Propositions for
$C^*$-algebras.\\
The next Lemma is interesting because it gives an operator space characterization (it involves computations on $2 \times 2$ matrices) of invertibility inside a von Neumann algebra.

\begin{lemma}\label{inver}
Let $\M$ be a von Neumann algebra and $x \in \M$, $\| x \| \leq
1$. Then, $x$ is invertible if and only if there exists $\alpha >0$
such that for any projection $y \in \M$, $$ \left\|
\left[ \begin{array}{c} x \\ y \end{array}\right] \right\|^2 \geq
\alpha + \| y \|^2 ~~\mbox{and}~~ \left\| \left[ \begin{array}{cc} x &
    y \end{array}\right] \right\|^2 \geq \alpha + \| y \|^2 \quad
(C)$$ If this holds, the maximum of the $\alpha$'s satisfying $(C)$
equals $\| x^{-1}\|^{-2}$ and $(C)$ holds for any $y\in \M$.
\end{lemma}
\begin{proof}
If $x\in \M$ is invertible and $y\in \M$ is arbitrary, by functional calculus, $x^*x\geq
\|x^{-1}\|^{-2}$, thus
$$\left\| \left[ \begin{array}{c} x \\ y \end{array}\right]
\right\|^2=\|x^*x+y^*y\| \geq \big\| \|x^{-1}\|^{-2}
+y^*y\big\|=\|x^{-1}\|^{-2}+ \|y\|^2.$$ Thanks to a similar argument
for the row estimate, we get that $(C)$ holds with
$\alpha=\|x^{-1}\|^{-2}$.

Assume that $(C)$ is satisfied.  Fix $\lambda\geq 0$ and let
$p_\lambda=\chi_{[0,\lambda]}(x^*x)\in \M$ be the spectral projection
of $|x|^2$ corresponding to $[0,\lambda]$. By the functional calculus
$ 1+\lambda \geq x^*x+ p_\lambda$ as $\|x\|\leq 1$.  Taking
$y=p_\lambda$ in the first part of $(C)$ gives,
$$1+\lambda \geq \left\| \left[ \begin{array}{c} x
    \\ p_\lambda \end{array}\right] \right\|^2\geq \alpha +
\|p_\lambda\|^2.$$ Hence $p_\lambda=0$ for $\lambda<\alpha$. Thus
$x^*x$ is invertible and $\big\|(x^*x)^{-1}\big\|\leq
\alpha^{-1}$. Similarly $xx^*$ must have the same property and $x$ is
left and right invertible hence invertible. In the polar decomposition
of $x=u|x|$, $u$ must be a unitary so that we finally get $\alpha\leq
\|x^{-1}\|^{-2}$ and the proof is complete.
\end{proof}

The next Proposition generalizes the well-known fact that a complete isometry between $C^*$-algebras sends unitaries to unitaries.

\begin{prop}\label{imageunit} Let $\A$, $\B$ be two 
$C^*$-algebras. Let $T:\A \to \B$ be a cb-isomorphism such that 
$\| T \|_{cb} \| T^{-1} \|_{cb} < \sqrt{2}$. Then, for any unitary $u\in A$, $T(u)$ is invertible and 
$$\| T(u)^{-1} \| \leq \frac{\| T^{-1}\|_{cb}}{\sqrt{2-\|T\|_{cb}^2\|T^{-1}\|_{cb}^2}}.$$
\end{prop}
\begin{proof}
Passing to biduals, we can assume that $\A$ and $\B$ are von Neumann
algebras as $x\in \B$ is invertible in $\B$ if and only if it is invertible in $\B^{**}$.
Replacing $T$ by $T/\| T \|_{cb}$, we may assume that $\| T \|_{cb}=1$
and $\| T^{-1} \|_{cb} <\sqrt 2$. 

Let $y \in\B$ with $\|y\|=1$, as $\|T(u)\|\leq 1$: $$ \left\| \left[\begin{array}{c} T(u) \\ y \end{array}\right]\right \|^2
\geq \frac{1}{\| T^{-1} \|_{cb}^2} \left\|\left[\begin{array}{c} u
  \\ T^{-1}(y) \end{array}\right]\right\|^2 \geq \frac{1+\|T^{-1}(y)\|^2}{\| T^{-1}
  \|_{cb}^2}\geq \frac{2}{\| T^{-1}
  \|_{cb}^2}.$$
Hence $T(u)$ satisfies $(C)$ with $\alpha=\frac{2}{\| T^{-1}
  \|_{cb}^2}-1>0$. Finally applying Lemma \ref{inver}, 
we obtain $$\| T(u)^{-1} \|^2 \leq \frac{1}{\alpha} \leq
\frac{\| T^{-1} \|_{cb}^2}{2-\| T^{-1} \|_{cb}^2}.$$
\end{proof}

This Lemma is folklore, we give a quick proof.

\begin{lemma}\label{unit} Let $\A$ be a unital $C^*$-algebra and  $x \in \A$ 
invertible. Then there exists a unitary $u \in \A$ such that  $\| x-u \| = \max \Big\{ \| x \| -1, 1-\frac{1}{\| x^{-1} \|} \Big\} .$
\end{lemma}
\begin{proof}
Write the polar decomposition of $x=u \vert x \vert$. As $x$ is
invertible, $\vert x \vert$ is strictly positive element of $\A$, so $u$ is a unitary of $\A$. Obviously, $\| x-u \| =\| \vert x \vert-1 \|
$. Seeing $\vert x \vert$ as a strictly positive function thanks to the functional calculus, it is not
difficult to conclude.
\end{proof}
 
The next Lemma is the key result to compute explicitly the defect of multiplicativity. As in Lemma \ref{inver}, operator space structure is needed.

\begin{lemma}\label{unitmult} Let $u,\, v$ be two unitaries in $\bB(H)$. 
 Let $x \in \bB(H)$ and $c \geq 1 $ such that 
$$\left\| \left[ \begin{array}{cc} u  & x  \\ -1 & v \end{array}\right] \right\| \leq c\sqrt{2},$$
then $\big\| x-uv \big\| \leq 2\sqrt{c^2-1} .$
\end{lemma}
\begin{proof}
Note first that $$ \left[ \begin{array}{cc} u^* & 0 \\ 0 &
    1 \end{array}\right] \left[ \begin{array}{cc} u & x \\ -1 &
    v \end{array}\right] \left[ \begin{array}{cc} 1 & 0 \\ 0 &
    v^* \end{array}\right]= \left[ \begin{array}{cc} 1 & u^*xv^* \\ -1
    & 1 \end{array}\right],$$ hence without loss of generality we can
assume that $u=v=1$. Take $h\in H$, then $$ \left\|
\left[ \begin{array}{cc} 1 & x \\ -1 & 1 \end{array}\right]
\left[ \begin{array}{c} -h \\ h \end{array}\right]\right\| \leq
c\sqrt{2} \left\| \left[ \begin{array}{c} -h
    \\ h \end{array}\right]\right\|.$$ Therefore $\| x(h)-h \|^2
+ 4\| h \|^2 \leq 4c^2\|h\|^2$, which implies $\| x-1 \| \leq
2\sqrt{c^2-1}$.
\end{proof}

We are now ready to prove the main result of this section.

\begin{thm}\label{defmult} Let $\A$, $\B$ be two unital
$C^*$-algebras. Let $T:\A \to \B$ be a cb-isomorphism with $T(1)=1$
and $\|T\|_{cb}\|T^{-1}\|_{cb}<\sqrt 2$, then
$$\|T^\vee\|_{cb}\leq 2 \sqrt{\Big(\|T\|_{cb}+ \frac {\mu (T)}{\sqrt 2}\Big)^2-1}+ \mu (T)(1+\|T\|_{cb}),$$
$$\| T-T^\star\|_{cb}\leq 2 \sqrt{\Big(\|T\|_{cb}+ \frac {\mu (T)}{\sqrt
      2}\Big)^2-1}+ 2\mu (T),$$
where $\displaystyle{\mu (T)= \max\Big\{ \|T\|_{cb}-1, 1-{\sqrt{\frac 2{\|T^{-1}\|_{cb}^2} -\|T\|_{cb}^2}}\Big\}}$.
\end{thm}
\begin{proof} We start with the defect of multiplicativity.
From the definition of the Haagerup tensor norm and the Russo-Dye Theorem,
it suffices to show that for any unitaries $u,\,v\in \bM_n(\A)$ we have 
$$\| T_n(uv)-T_n(u)T_n(v)\|_{\bM_n(\B)}\leq 2 \sqrt{\Big(\|T\|_{cb}+ \frac {\mu (T)}{\sqrt 2}\Big)^2-1}+ \mu (T)(1+\|T\|_{cb}),$$
where $T_n=Id_{\bM_n}\tens T$. Without loss of generality, we can assume $n=1$.

Let $u,\,v \in \A$ unitaries, as $\left\| \left[ \begin{array}{cc} u  & uv  \\ -1 & v \end{array}\right] \right\|=\sqrt{2}$, we get
$$\left\| \left[ \begin{array}{cc} T(u) & T(uv) \\ -1 &
    T(v) \end{array}\right] \right\| \leq \|T\|_{cb}\sqrt{2}.$$ From
Lemma \ref{unit} and Proposition \ref{imageunit} we deduce that there
are unitaries $u',\,v'\in \B$ with $\|T(u)-u'\|\leq \mu(T)$ and $\|T(v)-v'\|\leq \mu(T)$. The triangular inequality gives
$$\left\| \left[ \begin{array}{cc} u' & T(uv) \\ -1 &
    v' \end{array}\right] \right\| \leq \|T\|_{cb}\sqrt{2}+ \mu (T).$$
Lemma \ref{unitmult} implies that $\|T(uv)-u'v'\|\leq 2 \sqrt{\Big(\|T\|_{cb}+ \frac {\mu (T)}{\sqrt 2}\Big)^2-1}$, so that we get the estimate using the triangular inequality once more.

 For the second estimate, as $T_n^\star=(T^\star)_n$ we may also assume $n=1$. 
Thanks to the Russo-Dye Theorem, we just need to check that for any $u\in \A$ unitary 
$$ \| T(u)-T(u^*)^*\|\leq 2 \sqrt{\Big(\|T\|_{cb}+ \frac {\mu (T)}{\sqrt
      2}\Big)^2-1}+ 2\mu (T).$$
Taking $v=u^*$ in the above arguments leads to  $\|T(uu^*)-u'v'\|\leq 2 \sqrt{\Big(\|T\|_{cb}+ \frac {\mu (T)}{\sqrt 2}\Big)^2-1}$. Hence 
$\|u'-v'^*\|\leq 2 \sqrt{\Big(\|T\|_{cb}+ \frac {\mu (T)}{\sqrt 2}\Big)^2-1}$ and we conclude using the triangular inequality.
\end{proof}


\section{A noncommutative Amir-Cambern Theorem}

\subsection{Perturbations of multiplications}

\begin{defi}\label{multi} Let $\X$ be an operator space. A bilinear map $\m: \X \times \X \to \X$ is called a \textit{multiplication on} $\X$ if it is associative and extends to the Haagerup tensor product $\X \otimes_{h} \X$.\\
We denote by $\m_\A$ the original multiplication on an operator
algebra $\A$.
\end{defi}

In the following, $H_{cb}^k(\A,\A)$ denotes the $k^{th}$ completely
bounded cohomology group of $\A$ over itself. We refer to \cite{SS}
for precise definitions.

The next proposition is the operator space version 
of Theorem 3 in \cite{RT} or Theorem 2.1 in \cite{J1}. It gives a precise form of small perturbations of 
the product on an operator algebra under cohomological conditions.

As before, the quantity $\| \m-\m_{\A} \|_{cb}$ is the cb-norm of
$\m-\m_{\A}$ as a linear map from $\A \otimes_{h} \A$ into $\A$.

\begin{prop}\label{P:3} Let $\A$ be an operator algebra satisfying $$H_{cb}^2(\A,\A)=H_{cb}^3(\A,\A)=0. \quad (\star)$$ Then there exist $\delta,\, C>0$ such that for every multiplication $\m$ on $\A$ satisfying $\| \m-\m_{\A} \|_{cb} \leq \delta$, there is a completely bounded linear isomorphism $\Phi:\A \to \A$ such that
 $$\| \Phi-\id_\A \|_{cb} \leq C \| \m-\m_{\A} \|_{cb} ~~ \mbox{and}
  ~~ \Phi(\m(x,y))=\Phi(x)\Phi(y).$$ If $\A$ is a von Neumann algebra,
  then $(\star)$ is automatically satisfied with values
 $\delta=1/11$ and $C=10$. Moreover if $\m$ satisfies $\m(x^*,y^*)=\m(y,x)^*$
for all $x,\,y\in \A$, then $\Phi(x^*)=\Phi(x)^*$ for all $x\in \A$.
 \end{prop}
\begin{proof} We only give a sketch as it only consists in adapting
 arguments of \cite{RT} (see also \cite{SS} chapter 7) or Theorem 2.1 of \cite{J1} to the operator space category. 

  In the bounded situation, one  has to apply an implicit function 
theorem (Theorem 1 in
 \cite{RT}) to the right spaces of  multilinear maps
 (Theorem 3 in \cite{RT}). This is done in details in Theorem 7.4.1 in 
\cite{SS} from 7.3.1. With the notation there (taking $\M=\A$) one simply need
to replace $\L^k(\M,\M)$ by their cb-version $\L^k_{cb}(\A,\A)$ which are
 obviously Banach spaces. The statement about
 $*$ is justified right after Theorem 7.4.1 in \cite{SS}.\\
If $\A$ is a von Neumann algebra, all completely bounded cohomology
 groups of $\A$ over itself vanish (see \cite{SS} chapter 4.3) and we
 can choose $K=L=1$ in the proof of Theorem 2.1 of \cite{J1}, which
 gives $\delta=11^{-1}$ (see the discussion after the proof of Theorem
 2.1 \cite{J1}) and a computable value of $C$. Now following notation of the proof of Theorem 2.1 \cite{J1}, the rational function $p$ satisfies $p(x) \leq 9.75x^2$, for $x$ small enough.  Hence the sequence $(\varepsilon_i)$ defined by $\varepsilon_{i+1}=p(\varepsilon_i)$ (and $\varepsilon_0=\| \m-\m_{\A} \|_{cb}$) verifies $\varepsilon_i \leq 9.75^{2^i-1}\varepsilon_0^{2^i}$. As $K=L=1$, we have $\| S_{i} \|_{cb} \leq \varepsilon_{i-1} + 2\varepsilon_{i-1} ^2$.
 Then, using the previous estimates of the $\varepsilon_i$'s we get  $\| W_{n} -I \|_{cb} \leq \exp(\sum_{i=1}^n \| S_{i} \|_{cb}) \leq 10\varepsilon_0.$
 (With notation of Theorem 2.1 \cite{J1}, the desired $\Phi$ is obtained as the limit of $(W_{n})_n$.)
\end{proof}

\subsection{Proofs of Theorems A and B}



The proof of Theorem A is a variant of the proof of Theorem B. For clarity
we postpone the quantitative estimate to the next Remark.

\smallskip

\begin{proof}[of Theorem B] As $d_{KK}$ is bounded, the statement is only interesting when $d_{cb}(\A,\B)$ is 
close to 1. The proof uses ideas from Corollary 7.4.2 in \cite{SS}.

Let $L:\A\to \B$ be a cb-isomorphism with $\|L\|_{cb}\leq 1$ and
$\|L^{-1}\|_{cb}\leq d_{cb}(\A,\B)(1+\epsilon)$.

Consider the bidual extension  still denoted by $L:\A^{**}\to \B^{**}$,
it remains a cb-isomorphism and satisfies the
same norm estimates.
 
We suppose $\A$ unital, we will treat the non-unital case afterwards.
 The first step is to unitize $L$. By Lemma \ref{imageunit}, $L(1)$ is invertible in $\B$. Let
 $S=L(1)^{-1}L$, then $S$ is unital and $$\|S\|_{cb}\leq \frac
 {\|L^{-1}\|_{cb}}{\sqrt{2- \|L^{-1}\|_{cb}^2}}$$ and
 $\|S^{-1}\|_{cb}\leq {\|L^{-1}\|_{cb}}$. Note also that $S(\A)=\B$.

 The second step is to make our cb-isomorphism selfadjoint. By Theorem \ref{defmult}, we have $\|S-S^\star\|_{cb}\leq
 2f_1\big(\|L^{-1}\|_{cb}\big)$ for some continuous function with
 $f_1(1)=0$. Let $T=\frac 12\big(S+S^\star\big)$. Then 
$T:\A^{**} \to \B^{**}$ is unital
 $*$-preserving, $\|T-S\|_{cb}\leq f_1\big(\|L^{-1}\|_{cb}\big)$ and 
 $T(\A)\subset\B$. So if
 $\|L^{-1}\|_{cb}$ is close enough to 1, $T$ is also a cb-isomorphism
 such that $T(\A)=\B$ with norm estimates $\|T\|_{cb}\leq f_2\big(\|L^{-1}\|_{cb}\big)$ and $\|T^{-1}\|_{cb}\leq f_2\big(\|L^{-1}\|_{cb}\big)$, for some 
continuous function at $1$ with $f_2(1)=1$.

Define on $\A^{**}$ 
a new multiplication by, for $x,\,y\in \A^{**}$ 
$$\m(x,y)= T^{-1}\big(T(x)T(y)\big).$$ The multiplication $\m$ is
associative and $*$-preserving.  It is obviously completely bounded
and clearly $$\|\m-\m_{\A^{**}}\|_{cb}\leq \|T^{-1}\|_{cb}\|T^{\vee}\|_{cb}.$$
Thus, the estimate in Theorem \ref{defmult} gives that
$\|\m-\m_{\A^{**}}\|_{cb}\leq f_3 \big(\|L^{-1}\|_{cb}\big)$ for some
continuous function $f_3$ with $f_3(1)=0$. If $\|L^{-1}\|_{cb}$ is
close enough to 1, we get from Proposition \ref{P:3}, that there is a
completely bounded $*$-preserving linear isomorphism $\Phi: \A^{**}\to
\A^{**}$ with $\|\Phi-\id_{\A^{**}}\|_{cb}\leq f_4\big(\|L^{-1}\|_{cb}\big)$ and
for $x,\,y\in \A^{**}$
$$\Phi^{-1}\big( \Phi(x)\Phi(y)\big)=\m(x,y)=T^{-1}(T(x)T(y)).$$
Note that necessarily $\Phi(1)=1$, and $\Phi$ is $*$-preserving.

Let $\pi=T\Phi^{-1}:\A^{**}\to \B^{**}$, it is a $*$-preserving
cb-isomorphism. Moreover, for $x,\,y\in \A^{**}$,
$\pi(xy)=\pi(x)\pi(y)$, hence $\pi$ is actually a $*$-isomorphism. Now we check that the $C^*$-algebras $\pi(\A)$ and $\B$ are close for the Kadison-Kastler distance inside $\B^{**}$.
 We have, for $a\in Ball(\A)$:
\begin{equation}\label{II}
 \| \pi(a)-T(a) \|\leq f_4\big(\|L^{-1}\|_{cb}\big),\end{equation} 
and as $T(\A)=\B$, for $b\in Ball(\B)$, we have
$$ 
\|b-\pi\big(T^{-1}(b)\big)\|\leq \|T^{-1}\|_{cb} f_4\big(\|L^{-1}\|_{cb}\big)=f_5\big(\|L^{-1}\|_{cb}\big).$$
From which one easily deduces $d_{KK}(\A,\B)\leq f_5\big(\|L^{-1}\|_{cb}\big)$, for some
continuous function $f_5$ with $f_5(1)=0$.

Now if $\A$ is non-unital, in the preceding proof, $L(1)$ is now invertible in $\B^{**}$ (here $1$ denotes the unit of $\A^{**}$), so $S(\A)=\B$ is not valid anymore. But  the inequality \eqref{II} above still holds and we deduce that for $a\in Ball(\A)$$$ \| \pi(a)-S(a) \|\leq (f_4+f_1)\big(\|L^{-1}\|_{cb}\big).$$ Now from Lemma \ref{unit}, there is a unitary $u$ in $\B^{**}$ such that 
$ \|u- L(1)^{-1}  \| \leq f_6\big(\|L^{-1}\|_{cb}\big)$ for some
continuous function $f_6$ with $f_6(1)=0$. Therefore
$$ \| \pi(a)-uL(a) \|\leq (f_4+f_1+f_6)\big(\|L^{-1}\|_{cb}\big). $$ 
Taking the adjoints we obtain $$ \| \pi(a)-L(a^*)^*u^* \|\leq (f_4+f_1+f_6)\big(\|L^{-1}\|_{cb}\big). $$
Write $a=xy$, for some $x$ and $y$ in the $Ball(\A)$, then
\begin{equation}\label{eq2}
 \| \pi(a)-uL(x)L(y^*)^*u^* \|\leq
 2(f_4+f_1+f_6)\big(\|L^{-1}\|_{cb}\big).\end{equation} As
$L(x)L(y^*)^*$ belongs to $Ball(\B)$, we conclude that the
$C^*$-algebra $\pi(\A)$ is nearly included in the $C^*$-algebra $u\B
u^*$. Let us prove the converse near inclusion. Let $b \in Ball(\B)$,
we can factorize $b=L(x)L(y^*)^*$ with $x,y \in \A$ such that $ \| x
\| \leq \|L^{-1}\|$ and $ \| y \| \leq \|L^{-1}\|$. From inequality \eqref{eq2}, we get
$$ \| \pi(xy)-uL(x)L(y^*)^*u^* \|\leq 2 \|L^{-1}\|^2 (f_4+f_1+f_6)\big(\|L^{-1}\|_{cb}\big). $$
Finally, $d_{KK}(\A,\B)\leq 2 \|L^{-1}\|^2 (f_4+f_1+f_6)\big(\|L^{-1}\|_{cb}\big)$.
\end{proof}

\begin{proof}[of Theorem A] When $\A$ is a separable nuclear $C^*$-algebra, this follows directly from Theorem 4.3 in \cite{C5}
and Theorem B. \\
When $\A$ is a
  von Neumann algebra, one does not need to go 
to the bidual $\A^{**}$ in the preceding proof (to apply Proposition \ref{P:3}), so that we directly conclude that $\pi:\A\to \B$ is a $*$-isomorphism. But we should mention another way (which improves theoretically our bound in the von Neumann algebras case): the second step in the proof is not necessary, just define directly a new multiplication using the cb-isomorphism $S$ (instead of $T$). Then $\pi$ is just an algebra isomorphism (not necessarily selfadjoint), but it is enough to conclude thanks to Theorem 3 in \cite{Ga}.
\end{proof}


\begin{rk}\label{quant}
We give a rough estimate for the constants in Theorems A and B.  With notation from the proof
of Theorem B, we start with $\|L\|_{cb}\leq 1$ and
$\delta=\|L^{-1}\|_{cb}-1$, then $$\|S\|_{cb} \leq \frac {1+\delta}
{(2-(1+\delta)^2)^{1/2}}$$ and $\|S^{-1}\|_{cb} \leq 1+\delta$. From
now, we assume that $\delta\leq \frac 1{10}$. One easily checks,
computing derivative that
 $$\mu(S)= \max\Big\{ \|S\|_{cb}-1, 1-{\sqrt{\frac 2{\|S^{-1}\|_{cb}^2} -\|S\|_{cb}^2}}\Big\}\leq 2\delta,$$
$$\|S\|_{cb}<1+ 3\delta.$$ Then $$\| S-S^\star\|_{cb}\leq
2\sqrt{(1+3\delta+\sqrt 2 \delta)^2-1}+4\delta \leq 10 \sqrt \delta=2f_1(\delta).$$
We get that $T$ is invertible as soon as $5\sqrt \delta < \frac 1
{1+\delta}$.  Let us now assume that $\delta<\frac 1{200}$ so that
$$\|T\|_{cb} \leq \|S\|_{cb} +\|T-S\|_{cb} \leq 1+ 6 \sqrt \delta,$$
 $$\|T^{-1}\|_{cb} \leq \frac{\|S^{-1}\|_{cb}}
{1-\|S^{-1}\|_{cb}\|S-T\|_{cb}}\leq \frac {1+\delta}
{1-5(1+\delta)\sqrt\delta}\leq 1+ 8\sqrt\delta,$$ Now we get
$\displaystyle{ \mu(T) \leq 40 \sqrt \delta}$. Thus basic estimates
lead to $\|T^{-1}\|_{cb}\|T^\vee\|_{cb}\leq 180\sqrt \delta$, so we
can choose $f_3(1+\delta)=180\sqrt \delta$ for $\delta<\frac 1
{200}$. We need $f_3(1+\delta)<1/11$ to apply Proposition \ref{P:3},
so we assume $\delta<2.10^{-7}$. Hence for the von Neumann algebras
case of Theorem A, we could choose $\varepsilon_0=2.10^{-7}$, but we
will improve this bound later.

As $C=10$ in Proposition \ref{P:3}, $f_4=10f_3$. Moreover $f_6(1+\delta)=\|L(1)^{-1}\|-1=3\delta$.
Finally we obtain $d_{KK,cb}(\A,\B)\leq 3620 \sqrt \delta$.

  We need $d_{KK}(\A,\B)<1/420000$ to conclude
 for separable nuclear $C^*$-algebras. Finally, Theorem A for separable nuclear $C^*$-algebras is true with $\varepsilon_0=3.10^{-19}$.
 
For von Neumann algebras, as explained in the proof of theorem A,  
we only need  to deal with $S$. Hence $\varepsilon_0=4.10^{-6}$ is enough to 
ensure $\Vert S^{-1} \Vert_{cb} \Vert S^{\vee} \Vert_{cb}\leq 88\sqrt \delta < 1/11$ and to get the conclusion.
\end{rk}

\begin{rk}  It is clear that if any two $C^*$-algebras satisfy $d_{cb}(\A,\B)<1+4.10^{-6}$, then  $A^{**}$ and $B^{**}$ are $*$-isomorphic. We should also note that if preduals $\M_*,\N_*$ of von Neumann algebras satisfy $d_{cb}(\M_*,\N_*)<1+4.10^{-6}$, then $\M_*$ and $\N_*$ are completely isometric.\end{rk}

\begin{rk}  If a $C^*$-algebra $\A$ satisfies hypothesis of Proposition \ref{P:3} (with some constants $C$ and $\delta$), then $\A$ satisfies the conclusion of Theorem A (with $\varepsilon_0$ depending on $C$ and $\delta$). In particular, Johnson showed in \cite{J1} that separable unital commutative
$C^*$-algebras satisfy $H^2_{cb}(\A,\A)=H^3_{cb}(\A,\A)=\{0\}$. Thus,
  the above proof recovers the classical Amir-Cambern theorem (but
  with a worse $\varepsilon_0$, it is known that $\varepsilon_0=1$ is
  optimal is the commutative case).
\end{rk}

\begin{rk}  Actually when  $\A$ is a von Neumann algebra, we have proved slightly more: for any $C^*$-algebra $\B$,  for any completely bounded
linear isomorphism $L: \A \to \B $ such that $\|L\|_{cb} \|L^{-1}\|_{cb}\leq 1+\epsilon$ with $\epsilon < 2.10^{-7}$, then there exists a surjective linear complete isometry $J: \A \to \B$ such that $ \| J-L \|_{cb} \leq 1808 \sqrt \epsilon$. (With notation of the preceeding proof, just take $J=u^*\pi$). \end{rk}

\begin{rk}  For two operator spaces $\X,\Y$, denote: $d_{2}(\X,\Y)=\inf \big\{ \| \id_{\mathbb{M}_2} \otimes T \| \| \id_{\mathbb{M}_2} \otimes T^{-1} \|
\big\}, $ where the infimum runs over all possible linear isomorphisms
$T:\X \to \Y$. If a von Neumann algebra $\A$ satisfies
$H^2(\A,\A)=H^3(\A,\A)=\{0\}$ (which is true for factors of type $I$,
type $II_\infty$, type $III$ and type $II_1$ with property gamma or
admitting a Cartan MASA, see \cite{SS}), then the conclusion of Theorem
A is valid with assumption on $d_{2}$ instead of $d_{cb}$ (because norms estimates
needed in section 2 only require computations on $2 \times 2$
matrices).
\end{rk}

 We now give a counter-example to Amir-Cambern theorem for general (nuclear) $C^*$-algebras. Let us recall the main result of \cite{CC}:
\begin{thm}\label{contreex} There exist a
 family $(\C_\theta)_{\theta \in [0,\pi[}$ of non separable
     $C^*$-subalgebras of $\mathbb B(\ell_2)$ containing $\mathbb
     K(\ell_2)$ such that $d_{cb}(\C_\theta,\C_\tau)\leq C
     |\theta-\tau|$, $d_{KK,\ell_2}(\C_\theta,\C_\tau)\leq C
     |\theta-\tau|$ for some $C>0$ but $\C_0$ is not isomorphic to
     $\C_\theta$ with $\theta\neq 0$.
\end{thm}
Actually the isomorphisms $\C_\theta\to \C_\tau$ are completely positive. This shows somehow that Theorem A is optimal in full
generality.
We conclude with the following application
\begin{cor}\label{evian} 
There exists two (non separable) non isomorphic $C^*$-algebras $\C$ 
and $\D$ such that $d_{cb}(\C, \D)=1$.
\end{cor}
\begin{proof}
Just take $\C= \oplus_{\theta \in \mathbb Q\cap [0,1[}\C_\theta $ and
    $\D= \oplus_{\theta \in \mathbb Q\cap ]0,1[}\C_\theta$. For any
    $\varepsilon>0$, there is a bijection $f:\mathbb Q\cap
    [0,1[\to\mathbb Q\cap ]0,1[$ with $|f(x)-x|<\varepsilon$, thus
        $d_{cb}(\C, \D)<C\varepsilon$. But $\C$ and $\D$ are not
        isomorphic. Indeed for any $\theta$, $\C_\theta$ has a trivial center, 
 hence there is a minimal central projection $p\in \C$ with $p\C=\C_0$ whereas
$q\D\neq \C_0$ for all minimal central projections $q\in \D$.
\end{proof}

\subsection{Stability of length}

For details on the notion of length, see \cite{Pi1} and
\cite{Pi2}. To prove that the length function is locally constant, the first step is to notice that the length is stable for
cb-close multiplications, then an application of Theorem
\ref{defmult} allows to conclude.

The next lemma is folklore in perturbation theory, we just sketch the proof.

\begin{lemma}\label{L:2} Let $S,T: \X \to \Y$ be two completely bounded linear maps between operator spaces such that $\tilde{T}:\X / \ker T \to \Y$ is a cb-isomorphism with $\Vert \tilde{T}^{-1} \Vert_{cb} \leq K$. If $\Vert T-S \Vert_{cb}<1/K$, then $\tilde{S}:\X / \ker S \to \Y$ is also a cb-isomorphism and $$\Vert \tilde{S}^{-1} \Vert_{cb} \leq \frac{K}{1-K\Vert T-S \Vert_{cb}}.$$
\end{lemma}
\begin{proof}
Let $y$ be in the unit ball of $\mathbb{M}_n(\Y)$. Then there exists $x_0$ in $\mathbb{M}_n(\X)$, $\Vert x_0 \Vert < K $ such that $T(x_0)=y$. Hence
$$\Vert y-S(x_0) \Vert< \alpha, $$ where $\alpha=K\Vert T-S \Vert_{cb}<1$.
Applying the same procedure to $\frac{1}{\alpha}( y-S(x_0))$, we obtain $x_1$ in $\mathbb{M}_n(\X)$, $\Vert x_1 \Vert < K $ such that
$$\Vert y-S(x_0+\alpha x_1) \Vert< \alpha^2$$
proceeding by induction, we obtain the result.
\end{proof}

From \cite{BLM} Theorem 5.2.1, we know that an operator space $\A$
endowed with a multiplication $\m$ (in the sense of Definition
\ref{multi}) is cb-isomorphic via an algebra homomorphism to an actual
operator algebra. As the length is invariant under algebraic
cb-isomorphisms, it makes sense to talk about the length of $\A$
equipped with $\m$. As $\max$ is functorial, equivalences of Theorem
4.2 in \cite{Pi1} remain true in the case of a completely bounded
multiplication (not necessarily completely contractive). We denote by
$\m^l$ the $l$-linear map defined (by associativity) on $\A^l$.

\begin{prop}\label{P:4}  Let $\A$ be a unital operator algebra of length at most $\ell$ and length constant at most $K$. Let $\m$ be another multiplication on $\A$ such that $\Vert \m^\ell-\m_{\A}^\ell \Vert_{cb} < 1/K$. Then $\A$ equipped with $\m$ has also length at most $\ell$.
\end{prop}
\begin{proof}
Let us denote by $T_l:\max(\A) ^{\otimes_h \ell} \to \A$
(resp. $S_\ell:\max(\A) ^{\otimes_h \ell} \to \A$) the completely bounded
linear map induced by the original multiplication $\m_\A$ (resp. by
the new multiplication $\m$) on the $\ell$-fold Haagerup tensor product
of $\max(\A)$ (i.e. $\A$ endowed with its maximal operator space
structure). The hypothesis that $\A$ has length at most $\ell$ and
length constant at most $K$ exactly means that $\tilde{T_\ell}:\max(\A)
^{\otimes_h \ell}/ \ker T_\ell \to \A$ is a cb-isomorphism with $\Vert
\tilde{T_\ell}^{-1} \Vert_{cb} \leq K$ (see \cite{Pi1} Theorem 4.2). But
$\Vert \m^\ell-\m_{\A}^\ell \Vert_{cb} < 1/K$ implies that $\Vert T_\ell-S_\ell
\Vert_{cb}<1/K$. By Lemma \ref{L:2}, $\tilde{S_\ell}:\max(\A) ^{\otimes_h
  \ell} / \ker S_\ell \to \A$ is also a cb-isomorphism, so the result
follows.
\end{proof}

\begin{proof}[of Theorem C]
As in the proof of Theorem B, let $L: \A \to \B$ a linear
cb-isomorphism with $\Vert L \Vert_{cb} \Vert L^{-1} \Vert_{cb} <
1+\epsilon$. We consider $S=L(1)^{-1}L$ and the multiplication $\m$ on
$\A$ defined by $$\m(x,y)=S^{-1}(S(x)S(y)).$$ Hence $$\Vert
\m^\ell-\m_{\A}^\ell \Vert_{cb} \leq \Vert S^{-1} \Vert_{cb} \Vert S^{\vee
  \ell} \Vert_{cb}.$$ Since $\Vert S^{-1} \Vert_{cb} \leq
1+f(\epsilon)$ and $\Vert S \Vert_{cb} \leq 1+f(\epsilon)$ with
$f(\epsilon)$ tending to $0$ when $\epsilon$ tends to $0$, by
Theorem \ref{defmult}, $\Vert S^{-1} \Vert_{cb} \Vert S^{\vee \ell} \Vert_{cb}< 1/K $, if
$\epsilon$ is small enough. (Naturally here, $S^{\vee \ell}$ denotes the $\ell$-linear map defined on $\A^l$ by $S^{\vee \ell}(x_1, \dots, x_\ell)=S(x_1 \dots x_\ell)-S(x_1) \dots S( x_\ell)$). Therefore by Proposition \ref{P:4},
$\A$ equipped with $\m$ has also length at most $\ell$. But $S$ is a
cb-isomorphic algebra isomorphism from $\A$ equipped with $\m$ onto
$\B$, so $\B$ has length at most $\ell$ as well.\\
For the quantitative estimates, clearly $\Vert S^{\vee \ell} \Vert_{cb} \leq \Vert S^{\vee } \Vert_{cb}\sum _{k=0}^{\ell -2} \Vert S \Vert_{cb}^k$,
hence to apply Proposition \ref{P:4} we need $\Vert S^{-1} \Vert_{cb} \Vert S^{\vee \ell} \Vert_{cb}\leq 88\sqrt \delta 2^{\ell -1}< 1/K$, which gives the explicit bound.
\end{proof}

\begin{rk} We have just proved that the length is stable under cb-isomorphisms with small bounds. More generally, we have the following general principle: any property which is stable under perturbation by cb-close multiplications is also stable under perturbation by cb-isomorphisms with small bounds.
\end{rk}

\textbf{Acknowledgements.} The authors would like to thank Stuart White for mentioning the special case of von Neumann algebras in Proposition \ref{P:3} and Roger Smith for his valuable comments. The second author is grateful to Narutaka Ozawa for hosting him at RIMS in Kyoto when this work started.

\hfill \noindent \textbf{\'Eric Ricard} \\
\null \hfill Laboratoire de Math\'ematiques Nicolas Oresme \\ \null \hfill Universit\'e de Caen Basse-Normandie \\ \null \hfill
14032 Caen Cedex, France  \\ \null \hfill\texttt{eric.ricard@unicaen.fr}

\medskip

\hfill \noindent \textbf{Jean Roydor} \\
\null \hfill Institut de Math\'ematiques de Bordeaux  \\ \null \hfill 
Universit\'e Bordeaux 1 \\ \null \hfill
 351 Cours de la Libération \\
\null \hfill 33405 Talence Cedex, France\\
 \null \hfill\texttt{jean.roydor@math.u-bordeaux1.fr}

\end{document}